\documentclass[leqno]{amsart}

\usepackage{latexsym,amssymb,amsthm,amsmath,amscd}

\theoremstyle{plain}  
\newtheorem{theorem}{Theorem}[section]

\numberwithin{equation}{section}

 \newtheorem*{Theorem A}{{\bf Theorem A}}

\newtheorem{proposition}{Proposition}[section]

\numberwithin{equation}{section}

\theoremstyle{remark}
\newtheorem{remark}{Remark}[section]

 \numberwithin{equation}{section}
\def\({\left( }
\def\){\right)}
\def\e{\eqref}
 \def\I{{\bf I}^{n+1}}
 \def\R{{\bf R}^{n+1}}
 
\begin{document}

\title[Isotropic geometry of production models] {Notes on isotropic geometry of production models}

\author[B.-Y. Chen]{Bang-Yen Chen}

\address{Chen: Department of Mathematics,	Michigan State University, East Lansing, Michigan 48824--1027, USA}
\email{bychen@math.msu.edu}

\author[S. Decu]{Simona Decu}

\address{Decu: Department of Applied Mathematics, The Bucharest University of Economic Studies, Romania}
\email{simona.decu@gmail.com}

\author[L. Verstraelen]{Leopold Verstraelen}

\address{Verstraelen: Section of Geometry, KU Leuven, Belgium}
\email{leopold.verstraelen@wis.kuleuven.be}

\begin{abstract} The production function is one of the key concepts of mainstream neoclassical theories in economics. The study of the shape and  properties of the production possibility frontier is a subject of great interest
in economic analysis. In this respect, Cobb-Douglas and CES production functions with flat graph hypersurfaces in Euclidean spaces are first studied in  \cite{V,VV}. Later, more general studies of production models were given in \cite{c0}-\cite{c6}  and \cite{CV,VV2}.
On the other hand,  from visual-physical experiences,  the second and third authors proposed in \cite{DV} to study production models via isotropic geometry. 
The purpose of this paper is thus to investigate important production models via isotropic geometry. 
\end{abstract}



\maketitle

\section{Introduction.}

In economics, a production function is a non-constant positive function that specifies the output of a firm, an industry, or an entire economy for all combinations of inputs. Almost all economic theories presuppose a production function, either on the firm level or the aggregate level. In this sense, the production function is one of the key concepts of mainstream neoclassical theories.

Let $\mathbb E^{m}$ denote the Euclidean $m$-space, i.e., the Cartesian $m$-space ${\bf R}^{m}$ endowed with the Euclidean metric.
To visualize a production function $f:D\subset {\bf R}^n\to {\bf R}$ defined on a domain $D$ of ${\bf R}^n$,  we usually embed this $n$-space as $(x_1,\ldots,x_n)$-space into ${\mathbb E}^{n+1}$ and consider the corresponding graph hypersurface \begin{align}\label{1.1}\Phi(f):=\{(x_1,\ldots,x_n,f(x_1,\ldots,x_n))\in {\mathbb E}^{n+1}: (x_1,\ldots,x_n)\in D\}.\end{align}
 With respect to the metric induced from the Euclidean metric on $\mathbb E^{n+1}$, \e{1.1} defines an isometrically embedded hypersurface in $\mathbb E^{n+1}$.
 
It is well known that the study of the shape and the properties of the production possibility frontier is a subject of great interest in economic analysis.
G.\,E. V\^{\i}lcu  proved in \cite{V} that a Cobb-Douglas production function has constant return to scale if and only if the corresponding graph hypersurface in $\mathbb E^{n+1}$ is flat. Later, this result was extended to CES production function in \cite{VV}. Both Cobb-Douglas and CES production functions are homogeneous. More general geometric studies of production models were given in \cite{c0}-\cite{c6} and \cite{CV,VV2}.  

Minimality condition of production hypersurfaces was first studied by the first author in \cite{c1}. It is proved in \cite{c0} that a 2-input homogenous production function is a perfect substitute if and only if the production surface is minimal. Recently,  minimality condition of Cobb-Douglas and CES production functions was investigated by X. Wang and Y. Fu in \cite{WF}. 
 
Besides Euclidean geometry on ${\bf R}^m$, there is another important geometry  on ${\bf R}^{m}$, called isotropic geometry.  Isotropic geometry provides one of the 27 Cayley-Klein geometries on ${\bf R}^3$. It is the product of the Euclidean space and the isotropic line equipped with a degenerate parabolic distance metric.  The isotropic space ${\bf I}^{n+1}$ is derived from ${\mathbb E}^{n+1}$ by substituting isotropic distance for usual Euclidean distance. 

Given $p,q\in {\bf R}^{n+1}$ the isotropic distance $d(p,q)$ is either the Euclidean distance of the orthographic projections onto ${\bf R}^n\times \{0\}$ if the projections are distinct, or otherwise simply the Euclidean distance. 
In several applied sciences, e.g., computer science and vision, it is natural to view the graph hypersurface $\Phi(f)$ in \e{1.1} as a hypersurface in $\Phi(f)\times{\bf R}$ instead of $\mathbb E^{n+1}$, which  treat $\Phi(f)$ as a subset of the isotropic space ${\bf I}^{n+1}$ (see  \cite{KV,PO,S}).
For an $n$-input production function $f$, the metric on $\Phi(f)$ induced from  $\I$ is given by  $g_{*}=dx_1^2+\cdots +dx_n^2$. Thus $(\Phi(f),g_*)$ is always a flat space. So, its Laplacian  is given by \begin{align}\label{1.2}\Delta=\sum_{j=1}^n \frac{\partial^2}{\partial^2 x_j}.\end{align}

Recently,   it was suggested by the second and third authors in \cite{DV} to study production models via isotropic geometry. 
The purpose of this paper is thus to investigate production models from the viewpoint of isotropic geometry. Several classification results in this respect are obtained.

\section{A brief review of isotropic geometry}

For later use, we provide a brief review of isotropic geometry from \cite{PO} (see also \cite{DV,KV}).  

Let $f:D\subset {\bf R}^n\to {\bf R}$ be a function defined on a domain $D$ of ${\bf R}^n$.
Consider the graph hypersurface $\Phi(f)$ defined by
\begin{align}\label{2.1}\{(x_1,\ldots,x_n,f(x_1,\ldots,x_n))\in {\bf R}^{n+1}: (x_1,\ldots,x_n)\in D\}.\end{align}
In this paper we use the following terminology.  Lines in the $x_{n+1}$-direction are called {\it isotropic lines}. $k$-planes containing an isotropic line are called {\it isotropic $k$-planes}. 
The projections in the isotropic $x_{n+1}$-direction onto ${\mathbb E}^n$ are called {\it top views}. Via this projection, the top views of isotropic lines and $k$-planes are points and $(k-1)$-planes, respectively.

In isotropic geometry, it is convenient to represent a point $p\in \I$ with the coordinate vector
$X=(x_1,\ldots,x_{n+1})\in \R$ by its top view $x=(x_1,\ldots, x_n)$ and the last coordinate
$x_{n+1}$. 
We will use $i=(0,\ldots,0,1)\in \R$  for the isotropic direction and we write
$$X=x+x_{n+1}i$$ with the understanding that in this combination $x=(x_1,\ldots, x_n,0)$.

A curve $X(s)\subset \I$ without isotropic tangents can be represented by 
\begin{align}\label{2.2}X(s)=x(s)+x_{n+1}(s)i\end{align}  with $s$ as {\it  isotropic arclength}, which is the Euclidean arclength of its top view curve $x(s)$.
The derivative vectors $X'(s)$ and $x'(s)$  satisfy $||X'(s)||_i =||x'(x)||=1$,  where $||\cdot ||_i$ denotes the isotropic norm. Second derivative with respect to $s$ yields the curvature
vector $X''(s)=x''(s)+x''_{n+1}(s)i$. Thus we have
\begin{align}\label{2.3}\kappa(s):=||X''(s)||_i =||x''(s)||,\end{align}
which is the  curvature of the isotropic curve at $X(s)$. In fact, $\kappa(s)$ is nothing but the Euclidean curvature
of the top view $x(s)$. 
\vskip.05in 

\noindent {\sc Case} (1): $\kappa\ne 0$.  The principal normal vector is defined as $E_2:={X''(s)}/{||X''(s)||_i}$ in this case, which is isotropically orthogonal to $E_1:=X'(s)$ and satisfies $E'_1=\kappa E_2$.
\vskip.05in 

\noindent {\sc Case} (2): {\it For $s_o$ with $\kappa(s_o)=0$}. We define the {\it $s$-curvature} as $\kappa_s(s_o):= x''_{n+1}(s_0)$.
\vskip.1in

The curvature theory of hypersurfaces in $\I$  can be found in \cite{PO} which is analogous to the Euclidean counterpart. For a function $f:D\subset {\bf R}^n\to {\bf R}$,
consider the graph hypersurface $\Phi(f)=\{x+f(x)i: x\in D\}$.  Let $X(s) = x(s) + f(x(s) )i$ be a curve on $\Phi(f)$ which is parameterized by an isotropic arclength s. For its tangent vectors $E_1(s)=T(s)$,  we have
\begin{align}\label{2.4} T(s) = X'(s) =x'(s) + \left<x'(s),\nabla f(x(s))\right>i, \end{align}
where $\left<\,\cdot \,,\cdot\,\right>$ denotes the scalar product on $\mathbb E^n$ and $\nabla f$ is the gradient of $f$. 
In the following, we suppress the argument $s$ whenever there are no confusion.

The curvature vector is 
\begin{align}\label{2.5}X'' = x'' + \left<x'',\nabla f(x)\right>i + ( x'^T\cdot (D^2 f(x)) x')i, \end{align}
where $D^2 f$ is the Hessian of $f$. The first two terms in \e{2.5} form a vector $\tilde S$ in the tangent hyperplane $T(s_o)$ of $\Phi(f)$ at $X(x_o)$. The isotropic length of this vector is
considered as the geodesic curvature $\kappa_g$  of the curve at $X(s_o)$. Because of $||\tilde S||_i=||X''||_i$, $\kappa_g$
 is the same as the isotropic curvature $\kappa$  of the curve $X(s)$ at $X(s_o)$ and the Euclidean curvature of its top view. 
 
 For $\kappa_g\ne 0$, we normalize $\tilde S$ to the side vector $S=\tilde S/\kappa_g$, and for
  $\kappa_g=0$, we set  $S=i$. The third term in \e{2.5} is the ``normal'', i.e., isotropic component
of $X''$. Its $s$-length is called the {\it normal curvature} $\kappa_n$. Denoting the top view
$x'(s)$ of the tangent vector $T$ by $t$. Consequently, one has
  \begin{equation}\begin{aligned} \label{2.6} &X''=\kappa_g S+\kappa_n i,\\ & \kappa_g=\kappa =||X''||_i =||x''||,\;\; \kappa_n=t^T\cdot (D^2 f(x)) t.\end{aligned}\end{equation}

It is well known  that the extremal values of the normal curvatures at point $p\in \Phi(f)$ are the eigenvalues of the Hessian $D^2 f(p)$. The corresponding directions are the associated normalized eigenvectors. Since the Hessian is symmetric, all principal curvatures $\kappa_1,\ldots,\kappa_n$ are real and there is an orthonormal basis of main directions $t_1,\ldots,t_n$ in $\mathbb E^n$. From $t_1,\ldots,t_n$ we have the principal curvature directions of $\Phi(f)\subset \I$  given by $T_j=t_j+\left<t_j,\nabla f(p)\right>i,\, j=1,\ldots,n$.

Without solving the characteristic equation $\det( D^2 f  -\lambda E) = 0$, one can read off the
elementary symmetric functions of the principal curvatures from the coefficients of the
characteristic polynomial. This leads to $n$ fundamental
curvatures $K_1,\ldots,K_n$. If we denote by $H^{i_1,\ldots ,i_s}$ the
determinant of the quadratic submatrix of $D^2f(p)$ obtained by taking in $D^2f$ only rows
and columns with indices $i_1,\ldots,i_s$, then we have
 \begin{equation}\begin{aligned} \label{2.7} K_j &=\frac{1}{\binom{n}{j}}\big(\kappa_1\cdots \kappa_j+\kappa_1\cdots \kappa_{j-1}\kappa_{j+1}+\cdots +\kappa_{n-j+1}\cdots \kappa_n\big),
 \\ & =\frac{1}{\binom{n}{j}}\big(H^{1,\ldots, j}+H^{1,\ldots, j-1,j+1}+\cdots +H^{n-j+1,\ldots, n}\big).\end{aligned}\end{equation}
In particular, one has the {\it isotropic mean curvature}
\begin{align}\label{2.8} K_1(p)=\frac{1}{n}{\rm trace}\,(D^2f(p))=\frac{1}{n}\Delta f(p),\end{align}
and the analogue of the Gaussian-Kronecker curvature, the  {\it relative curvature}
\begin{align}\label{2.9} K_n(p)=\det\,(D^2f(p)).\end{align}

\begin{remark} It follows from \e{2.8} that the graph hypersurface $\Phi(f)$ of production function $f$ is isotropic minimal in $\I$ if and only if $f$ is harmonic, i.e., $\Delta f=0$.
\end{remark}

\section{Some important production functions in economics}

There are two special classes of production functions that are often analyzed in microeconomics and macroeconomics; namely, homogeneous and homothetic production functions. A production function $ f(x_{1},\cdots,x_{n})$ is said to be  {\it homogeneous of degree} $d$ or {\it $d$-homogeneous}, if  
  \begin{align}\label{3.1}f(tx_{1},\ldots,tx_{n}) = t^{d}f(x_{1},\ldots,x_{n})\end{align}
holds for each $t\in \mathbb R$ for which \eqref{3.1} is defined.   A homogeneous function of degree one is called  {\it linearly homogeneous}.

If $d>1$, the homogeneous function exhibits increasing returns to scale, and it exhibits decreasing returns to scale if $d<1$. If it is homogeneous of degree one, it exhibits constant returns to scale.  Constant returns to scale is the in-between case. 

A {\it homothetic  function} is a production function of the form:  
\begin{align}\label{3.2} f=F(h(x_1,\ldots,x_n)),\end{align}
where $h(x_1,\ldots,x_n)$ is a homogeneous function of any given degree and $F$ is a monotonically increasing function.

In economics, an {\it isoquant} is a contour line drawn through the set of points at which the same quantity of output is produced while changing the quantities of two or more inputs. 
Homothetic functions are functions whose marginal technical rate of substitution (the slope of the isoquant) is homogeneous of degree zero. 

The $n$-input Cobb-Douglas  production function can be expressed as
\begin{align}\label{3.3} f =\gamma x_{1}^{\alpha_{1}}\cdots x_{n}^{\alpha_{n}},\end{align}
where $\gamma$ is a positive constant and $\alpha_{1},\ldots,\alpha_{n}$ are nonzero constants. 
The Cobb-Douglas production function is especially notable for being the first time an aggregate or
economy-wide production function was developed, estimated, and then presented to the profession for analysis. It gave a landmark change in how economists approached macroeconomics. 

The $n$-input CES  production functions (or ACMS production function \cite{ACMS}) are given by
\begin{align}\label{3.4} f=\gamma \left(\sum_{i=1}^{n}a_{i}^{\rho}x_{i}^{\rho}\right)^{\! \frac{d}{\rho}},\end{align}
where $a_{i},d,\gamma, \rho$ are nonzero constants. The CES production functions are of great interest in economy because of their invariant characteristic; namely, the elasticity of substitution between the parameters is constant on their domains. 
Obvious, both Cobb-Douglas and CES production functions are homogeneous.

In economics, goods that are completely substitutable with each other are called perfect substitutes. They may be characterized as goods having a constant marginal rate of substitution. Mathematically, a production function is  a perfect substitute if it is of the form
\begin{align}\label{3.5} f(x_{1},\ldots,x_{n})= \sum_{i=1}^{n}a_{i}x_{i}\end{align}for some nonzero constants $a_{1},\ldots,a_{n}$.

\section {Minimality of production models.}

For a production function $f: D\subset {\bf R}^n\to {\bf R}$, the study of minimality of the
graph hypersurface 
  \begin{align}\label{4.1}\Phi(f):=\{(x_1,\ldots,x_n,f(x_1,\ldots,x_n))\in {\mathbb E}^{n+1}: (x_1,\ldots,x_n)\in D\}\end{align}
in $\mathbb E^{n+1}$ with respect to the Euclidean metric on ${\mathbb E}^{n+1}$ was initiated by the first author  in \cite{c0} and later in \cite{WF} by Wang and Fu. The following results were proved in \cite{c0,WF}.

 \begin{theorem} \label{T:4.1} \cite{c1} A 2-input homogeneous production function is a perfect substitute  if and only if the production surface is a minimal surface in $\mathbb E^3$.
 \end{theorem}
  \begin{theorem} \label{T:4.2} \cite{WF} There does not exist a minimal Cobb-Douglas production hypersurface in $\mathbb E^{n+1}$.
 \end{theorem}

 \begin{theorem} \label{T:4.3} \cite{WF} An $n$-input CES production hypersurface
in $\mathbb E^{n+1}$ is minimal if and only if the production function is a perfect substitute.
 \end{theorem}

Next, we discuss the minimality condition of production models in the isotropic space $\I$, instead of ${\mathbb E}^{n+1}$.

For 2-input functions, we give the following simple geometric characterization of perfect substitute via isotropic geometry.

\begin{proposition} \label{P:4.1} Let $f$ be a 2-input linearly homogeneous production function. Then $f$ is a perfect substitute  if and only if the production surface is isotropic minimal in the isotropic 3-space ${\bf I}^3$.\end{proposition}
\begin{proof} Assume that $f(x_1,x_2)$ be a linearly homogeneous production function. Then it follows from the Euler Homogeneous Function Theorem that $f$ satisfies
\begin{align} \label{4.2} x_1 f_{x_1}+x_2 f_{x_2}=f,\end{align}
where $f_{x_j}$ denotes the partial derivative of $f$ with respect to $x_j$. By taking the partial derivatives of \e{4.2} with respect to $x_1,x_2$, respectively, we find
\begin{equation}\begin{aligned}\label{4.3} &x_1 f_{x_1 x_1}+x_2 f_{x_1 x_2}=0,
\\& x_1 f_{x_1 x_2}+x_2 f_{x_2 x_2}=0.\end{aligned}\end{equation}
From \e{4.3} we get
\begin{align} \label{4.4} \Delta f=f_{x_1x_1}+ f_{x_2x_2}=-f_{x_1 x_2}\left(\frac{x_1^2+x_2^2}{x_1 x_2}\right),\end{align}
Therefore, it follows from \e{2.8} and \e{4.4} that if the production surface is isotropic minimal in ${\bf I}^3$, then $f_{x_1x_2}=0$, which implies that $f(x_1,x_2)=p(x_1)+q(x_2)$ for some functions $p$ and $q$. Since $f$ is assumed to be linearly homogeneous, we conclude that $f$ is a perfect substitute.

The converse is trivial. 
\end{proof}

\begin{remark} Contrast to Theorem \ref{T:4.1}, Proposition \ref{P:4.1} is false if the homogeneous production function $f$ is nonlinear. Two simple examples are the degree 2 homogeneous production functions given by $f=x_1 x_2$ and $k=x^2- y^2$. It is easy to verify that the graph surfaces  of $f$ and $k$ are isotropic minimal  in ${\bf I}^3$, but  not perfect substitutes. 
\end{remark}

Theorem 4.2 states that there do not exist minimal Cobb-Douglas production hypersurfaces in $\mathbb E^{n+1}$. On the other hand, the next  result states that there do exist isotropic minimal Cobb-Douglas production models in $\I$.

 \begin{proposition} \label{P:4.2} An $n$-input Cobb-Douglas production hypersurface is isotropic minimal in $\I$ if and only if the production function is of the form $f=\gamma x_1\cdots x_n$ for some nonzero constant $\gamma$.
 \end{proposition}
 \begin{proof} Follows from \e{2.8} and \e{3.3}.
   \end{proof}

Analogous to Theorem \ref{T:4.3}, we have the following simple geometric characterization of $n$-input perfect substitute in term of isotropic minimality.

  \begin{proposition} \label{P:4.3} An $n$-input CES production hypersurface is isotropic minimal in $\I$ if and only if the production function is a perfect substitute.
  \end{proposition}
 \begin{proof} Let $f$ is an $n$-input CES production function defined by
 \begin{align}\label{4.5} f(x_1,\ldots,x_n)=\gamma \left(\sum_{i=1}^{n}a_{i}^{\rho}x_{i}^{\rho}\right)^{\! \frac{d}{\rho}},\end{align}
where $a_{i},d,\gamma, \rho$ are nonzero constants. By applying a straight-forward computation we find 
\begin{equation}  \begin{aligned}\label{4.6} &f_{x_j}=\gamma d a^\rho_j x_j^{\rho-1}  \(\sum_{i=1}^{n}a_{i}^{\rho}x_{i}^{\rho}\)^{\! \frac{d}{\rho}-1},
  \\&f_{x_jx_j}=\gamma d a^\rho_j x_j^{\rho-2}  \(\sum_{i=1}^{n}a_{i}^{\rho}x_{i}^{\rho}\)^{\! \frac{d}{\rho}-2}\((\rho-1)\sum_{i=1}^{n}a_{i}^{\rho}x_{i}^{\rho}+(d-\rho)a_j^\rho x_j^\rho   \),
 \end{aligned}\end{equation} for $j=1,\ldots,n$.
After a direct computation we obtain from \e{4.6} that the Laplacian of $f$ satisfies
\begin{equation}  \begin{aligned}\label{4.7} \Delta f &=\frac{\gamma d \(\sum_{i=1}^{n}a_{i}^{\rho}x_{i}^{\rho}\)^{\! \frac{d}{\rho}-2}}{x_1^2\cdots x_n^2}\! \Big\{\!(d\! -\! 1)\big(a_1^{2\rho} x_1^{2\rho}x_2^2\cdots x_n^2 +\cdots +a_n^{2\rho} x_1^2\cdots x_{n-1}^2 x_n^{2\rho} \big)
\\&\hskip.3in  +(\rho-1)\big\{ (a_1 a_2 x_1 x_2)^\rho (x_1^2+x_2^2) x_3^2\cdots x_n^2 +\cdots 
\\&\hskip.5in +  x_1^2\cdots x_{n-2}^2(a_{n-1}a_n x_{n-1}x_n)^\rho (x_{n-1}^2 +x_n^2)
\big\}  \Big\},
  \end{aligned}\end{equation}
which implies that the graph hypersurface $\Phi(f)$ is isotropic minimal in $\I$ if and only if $d=\rho=1$, i.e., $f$ is a perfect substitute.
  \end{proof}

\section{Remarks on production models in $\I$ with null relative curvature}

It follows from \e{2.9} that the graph hypersurface $\Phi(f)$ of a production function $f$ has null relative curvature if and only if the production function is a solution to the homogeneous Monge-Amp\`ere equation. Hence, by \cite[Proposition 2.1]{c0}, the results given in the last section of \cite{c6} can be rephrased as the following.

\begin{theorem} \label{T:5.1}Let $f=F(h(x_{1},\ldots,x_{n}))$ be a homothetic production function such that $h$ is a homogeneous function with $\deg h\ne 1$. Then the graph hypersurface $\Phi(f)$ of $f$ in $\I$ has null relative curvature if and only if either 
\vskip.05in

 {\rm (i)}  $h$ satisfies the homogeneous Monge-Amp\`ere equation $\det(h_{ij})=0$ or 

\vskip.05in

{\rm (ii)} up to constants, $f=F\circ h$ is a linearly homogeneous function.
  \end{theorem}

For Cobb-Douglas and CES production models, we have the following.

\begin{theorem} \label{T:5.2} Let $h(x_{1},\ldots,x_{n})$ be a Cobb-Douglas production function. Then the graph hypersurface $\Phi(f)$ of the homothetic production function $f=F\circ h$ in $\I$ has null relative curvature if and only if both $F$ and $h$ are linear.  \end{theorem}

\begin{theorem} \label{T:5.3} Let $h(x_{1},\ldots,x_{n})$ be a CES production function defined by \e{3.4}. Then the graph hypersurface of the homothetic production function $f=F\circ h$ in $\I$ has null relative curvature  if and only if either
\begin{enumerate}
\item $\rho=1$, or 

\item  $f$ has constant return to scale.\end{enumerate}
  \end{theorem}

For 2-input homothetic production functions, we have the following.

\begin{theorem} \label{T:5.4} Let $f(x,y)=F(h(x,y))$ be a homothetic production function. Then 
 the graph surface of  $f$ has null relative curvature if and only if either
 \begin{enumerate}

\item  $f(x,y)$ is  linearly homogeneous, or
 
\item the inner function $h(x,y)$ is a perfect substitute.
\end{enumerate} 
  \end{theorem}

\section{Remark on flat production models in $\I$}

Homogeneous production functions with flat graph hypersurfaces in ${\mathbb E}^{n+1}$ were classified by the first author and G. E. V\^{\i}lcu  as follows.

\begin{theorem} \label{T:6.1} \cite{CV} Let $f$ be a homogeneous production function with $\deg f=d$.  Then the graph hypersurface of $f$ in $\mathbb E^{n+1}$ is flat if and only if either 
\begin{enumerate}
\item  $f$ has constant return to scale, or

\item $f$ is of the form
\begin{align}\label{f}f=\left(c_1x_{1}+c_2x_{2}+\cdots+c_{n} x_{n}\right)^d,\;\; d\ne 1,\end{align} for some real constants $c_{1},\ldots,c_{n}$.
\end{enumerate}
 \end{theorem}

Finally, we discuss isotropic flat production models in $\I$.
First, we remark that for an $n$-input production function $f$, it follows from \e{2.6} that all curves on $\Phi(f)$ passing through a given point  with a given tangent have the same normal curvature. 

For a 2-dimensional subspace $\pi$ of the tangent space $T_p(\Phi(f))$ of the graph hypersurface $\Phi(f)$ at a point $p$, we define the {\it isotropic sectional curvature} $K_i(\pi)$ to be the product of the maximum and minimum normal curvatures with respect to the directions in $\pi$. A graph hypersurface $\Phi(f)$ in $\I$ is said to be {\it isotropic flat} if its isotropic sectional curvatures vanish identically.

Analogous to Theorem \ref{T:6.1}  we have the following.

\begin{theorem} \label{T:6.2} Let $f$ be a homogeneous production function.  Then the graph hypersurface of $f$ in $\I$ is isotropic flat if and only if either 
\begin{enumerate}
\item $f$ has constant return to scale, or

\item $\deg f\ne 1$ and $f$ is a power of a perfect substitute.
\end{enumerate}
 \end{theorem}
\begin{proof} Let $f(x_1,\ldots,x_n)$ be a homogeneous production function. Let us consider the graph hypersurface $\Phi(f)$ of $f$ and the plane section $\pi_{jk}={\rm Span}\{\frac{\partial}{\partial x_j},\frac{\partial}{\partial x_k}\}$ with $1\leq j\ne k\leq n,$ at a point on $\Phi(f)$  in $\I$. Then the isotropic sectional curvature $K_i(\pi_{jk})$ is obtained by taking $D^2f$ only from rows and columns with indices $j,k$. Therefore, if the graph hypersurface $\Phi(f)$ is isotropic flat in $\I$, then we have $$f_{x_j x_j}f_{x_k x_k}-f_{x_j x_k}^2=0$$
for $1\leq j\ne k\ne n$.
Consequently, we may apply the same argument as in the proof of Theorem 1.1 of \cite{CV} to conclude the theorem.
\end{proof}

\section{Concluding remarks}

By imposing various curvature conditions on graph hypersurfaces, we observe in the last three sections several differences and similarities between graph hypersurfaces in $\mathbb E^{n+1}$ and graph hypersurfaces in $\I$.
In the following, we discuss some differences between graph hypersurfaces in $\mathbb E^{n+1}$ and graph hypersurfaces in $\I$ from the viewpoint of finite type theory in the sense of the first author \cite{c1984,book,c7}.

Analogous to the theory of finite type submanifolds in $\mathbb E^{n+1}$ (cf. \cite{book,c7}), we call a graph hypersurface $\Phi(f)$ in $\I$ to be of {\it finite type} if all coordinate functions of $\Phi(f)$ in $\I$ are finite sums of eigenfunctions of the Laplacian $\Delta$ on $\mathbb E^n$. It is clear from \e{2.1} that a graph hypersurface $\Phi(f)$ in $\I$ is of finite type if and only if the function $f$ is of finite type. Since there exist many finite type functions, there are abundant examples of finite type graph hypersurfaces in $\I$, a phenomenon contrast to finite type hypersurfaces in Euclidean spaces (cf. \cite{c7}).

Similar to Euclidean submanifolds,  we call  a graph hypersurface $\Phi(f)$ in $\I$ {\it isotropically biharmonic} if the position function ${\bf x}$ of $\Phi(f)$ in $\I$ is biharmonic, i.e., it satisfies the biharmonic equation: \begin{align}\label{7.1} \Delta^2 {\bf x}=0.\end{align}

It follows from \e{2.1} and \e{7.1} that a graph hypersurface $\Phi(f)$ in $\I$ is  isotropically  biharmonic if and only if the function $f$ is biharmonic. Obviously, there exist many biharmonic functions which are not harmonic. Consequently, there are many graph hypersurfaces in $\I$ which are  isotropically  biharmonic, but not harmonic. This phenomenon is quite different from Euclidean biharmonic submanifolds. In fact, the first author conjectured in 1981 that minimal submanifolds are the only biharmonic submanifolds in Euclidean spaces. This biharmonic conjecture remains unsettled after more than three decades (see \cite{c5,c7} for  the most recent development on biharmonic conjecture).

\end{document}